\documentclass[12pt]{article}
\usepackage{amssymb,amsfonts,amsthm}
\usepackage{amsmath,amsthm,amssymb}
\usepackage{amscd}
\usepackage{amsfonts}
\usepackage{color}
\newcommand{\con}{\operatorname{Cont}}

\newcommand{\non}{\operatorname{Non}}
\newcommand{\lin}{\operatorname{Lin}}
\newcommand{\iso}{\operatorname{Isot}}
\newcommand{\var}{\operatorname{var}}

\newtheorem{theorem}{Theorem}[section]

\newtheorem{cor}[theorem]{Corollary}

\newtheorem{fact}[theorem]{Fact}
\newtheorem{lemma}[theorem]{Lemma}

\newtheorem{question}{Question}
\newtheorem{claim}{Claim}

\begin{document}

\title{Lee monoids are non-finitely based while the sets of their isoterms are finitely based}
\author{ Olga Sapir
\footnotesize\em email: olga.sapir@gmail.com\\}
\date{}
\maketitle

\begin{abstract} We establish a new sufficient condition under which a monoid is non-finitely based and apply this condition  to Lee monoids $L_\ell^1$, obtained by adjoining an identity element to the semigroup generated by two idempotents $a$ and $b$ subjected to the relation $0=abab \cdots$ (length $\ell$).

We show that every monoid $M$ which generates a variety containing $L_5^1$ and is contained in the variety generated by $L_\ell^1$
for some $\ell \ge 5$ is non-finitely based. We establish this result by analyzing $\tau$-terms for $M$ where $\tau$ is certain non-trivial congruence  on the free semigroup, that is, we analyze words $\bf u$ with the property that  ${\bf u} \tau {\bf v}$ whenever $M$ satisfies an identity ${\bf u} \approx {\bf v}$.

We also show that if  $\tau$ is the trivial congruence on the free semigroup and $\ell \le 5$ then the $\tau$-terms (isoterms) for $L_\ell^1$ carry no information about the non-finite basis property of $L_\ell^1$.

\vskip 0.1in


\end{abstract}

\section{Introduction}

An algebra  is said to be {\em finitely based} (FB) if there is a finite subset of its identities from which all of its identities may be deduced.
Otherwise, an algebra is said to be {\em non-finitely based} (NFB).
Throughout this article, elements of a countably infinite alphabet $\mathfrak A$ are called {\em variables} and elements of the free monoid $\mathfrak A^*$  and free semigroup $\mathfrak A^+$ are called {\em words}.

In 1968, P. Perkins \cite{P} found the first sufficient condition under which a monoid (semigroup with an identity element) is NFB.
By using this condition, he constructed the first two examples of finite NFB semigroups.  The first example was the 6-element Brandt monoid and the second example was the 25-element monoid obtained from the set of words
$W= \{abtba, atbab, abab, aat\}$ by using the following construction attributed to Dilworth.

 Let  $W$ be a set of words in the free monoid ${\mathfrak A}^*$. Let $S^1(W)$ denote the Rees quotient  over the ideal of  ${\mathfrak A}^*$ consisting of all words that are not subwords of words in $W$. For each set of words $W$, the semigroup $S^1(W)$ is a monoid with zero whose nonzero elements are the subwords of words in $W$. We say that $W$ is (\textit{non}-)\textit{finitely based} if the monoid $S^1(W)$ is (non-)finitely based.

 For certainty, we regard monoids here as semigroups, that is, algebras with one operation.  It is mentioned in \cite{MV}  that the finite basis property of a monoid does not depend on whether it is considered as an algebra with one operation or two operations. Indeed, if  an identity  can be derived from a finite set of identities $\Sigma$ by using the substitutions  $\Theta: \mathfrak A \rightarrow \mathfrak A^*$ then it can be also derived from a finite set of identities $\Sigma'$ by using only the substitutions  $\Theta: \mathfrak A \rightarrow \mathfrak A^+$.
(Just take $\Sigma'$ to be the set of all identities obtained by deleting variables in the identities in $\Sigma$.)

If $\tau$ is an equivalence relation on the free semigroup $\mathfrak A^+$ then we say that a word ${\bf u}$ is  a {\em $\tau$-term} for a semigroup $S$ if ${\bf u} \tau {\bf v}$ whenever $S$ satisfies ${\bf u} \approx {\bf v}$.
If $\tau$ is the equality relation on $\mathfrak A^+$ then  $\tau$-term is an {\em isoterm} \cite{P} for $S$.
 We use $\var S$ to refer to the variety of semigroups generated by $S$.
The following result of M. Jackson gives us the fundamental connection between monoids of the form $S^1(W)$ and  isoterms for monoids.

\begin{fact} \label{prec}  \cite[Lemma 3.3]{MJ}
Let $W$ be a set of words and $M$ be a monoid.
Then  $\var M$ contains $S^1(W)$  if and only if every word in $W$ is an isoterm for $M$.
\end{fact}

Given a monoid $M$ we use  $\iso(M)$ to denote the set of all words in $\mathfrak A^*$ that are isoterms for $M$. Using Fact \ref{prec} it is easy to show
that $W=\iso(M)$ is the largest subset of $\mathfrak A^*$ such that $S^1(W)$ is contained in  $\var M$ (see Fact 8.1 in \cite{OS1}).

A locally finite algebra is said to be {\em inherently not finitely based} (INFB) if any locally finite variety containing it is NFB.
According to Proposition 7 in \cite{MS}, a finite semigroup $S$ is INFB  if and only if every Zimin word (${\bf Z}_1=x_1, \dots, {\bf Z}_{k+1} = {\bf Z}_kx_{k+1}{\bf Z}_k, \dots$)  is an isoterm for $S$.  This result of M. Sapir together with Proposition 3 in \cite{MS} imply that if $M$ is a finite INFB monoid then the set  $\iso(M)$ is NFB.  Proposition 7 in \cite{MS} also implies that the Brandt monoid is INFB and consequently, the set of its isoterms is non-finitely based.

For the majority of the aperiodic monoids which are known to be NFB but not INFB, their non-finite basis property can be established by
exhibiting a certain finite set of words $W$, a certain set of identities $\Sigma$ (without any bound on the number of variables involved) and proving the following statement:

 $\bullet$ If a monoid $M$ satisfies all identities in $\Sigma$ and all the words in $W$ are isoterms for $M$, then $M$ is NFB.

If the non-finite basis property of a monoid $M$ is established by a sufficient condition of this form, then evidently, the set $\iso(M)$ is also NFB.

We say that a word $\bf u$ has {\em the same type} as $\bf v$ if $\bf u$ can be obtained from $\bf v$ by changing the individual exponents of variables. For example, the words $x^2yxzx^5y^2xzx^3$ and
 $xy^2x^3zxyx^2zx$  are of the same type.
In this article, we present a new sufficient condition (see Theorem \ref{main} below) under which a monoid is non-finitely based.  Theorem \ref{main}
exhibits a certain finite set of words $W$, a certain set of identities $\Sigma$ (without any bound on the number of variables involved) and
states the following:

 $\bullet$ If a monoid $M$ satisfies all identities in $\Sigma$ and every word in $W$ can form an identity of $M$ only with a word of the same type, then $M$ is NFB.

Recently,  E. Lee suggested to investigate the finite basis property of semigroups
\[
L_\ell = \langle a,b \mid aa=a, bb=b, \underbrace{ababab\cdots}_{\text{length }\ell} = 0 \rangle, \quad \ell \geq 2
\]
and the monoids $L_\ell^1$ obtained by adjoining an identity element to $L_\ell$.

The 4-element semigroup $L_2 =A_0$ is  long known to be finitely based \cite{1980}.
W. Zhang and Y. Luo proved \cite{WTZ1} that the 6-element semigroup $L_3$ is NFB and E. Lee generalized this
result into a sufficient condition  \cite{EL} which implies that for all $\ell \geq 3$, the semigroup $L_\ell$ is NFB  \cite{EL1}.

 As for the monoids $L_\ell^1$, the 5-element monoid $L_2^1$ was also proved to be FB by C. Edmunds \cite{1977}, while the 7-element monoid $L_3^1$ is recently shown to be NFB by W. Zhang \cite{WTZ}. E. Lee conjectured that the monoids $L^1_\ell$ are NFB for  all $\ell \geq 3$. Theorem \ref{main} implies that for each $\ell \ge 5$ the monoid $L^1_\ell$ is NFB.  This leaves the 9-element monoid $L_4^1$ the only unsolved case in the finite basis problem for the monoids $L_\ell^1$.

We prove Theorem \ref{main} by using the general method in \cite{OS}. This general method can be used to establish the majority of existing sufficient conditions under which a semigroup is NFB. In particular, it can also be used to reprove the sufficient condition of Lee \cite{EL} which implies that  for all $\ell \geq 3$, the semigroup $L_\ell$ is NFB. (The proof is the same as the one of Theorem 5.2 in \cite{OS} but it uses Lemma 14 in \cite{EL} instead of Lemma 13 in \cite{EL2}.)  Thus, this method can be used to establish the non-finite basis properties of both: Lee semigroups and Lee monoids.

In Section 7 we introduce monoids of the form $S^1_\tau(W)$ and
show that both Lee monoids and the monoids of the form $S^1(W)$ can be viewed as special cases of this general construction.  We also generalize Fact \ref{prec} into Lemma \ref{prec1} which  gives us the
 connection between monoids of the form $S^1_\tau(W)$ and $\tau$-terms when $\tau$ is not necessarily the equality relation on $\mathfrak A^+$.

\section{A sufficient condition under which a monoid is non-finitely based}

If $\bf u$ is a word and $x \in \con({\bf u})$ then an {\em island} formed by $x$  in {\bf u} is a maximal subword of $\bf u$ which is a power of $x$. For example, the word $xyyx^5yx^3$ has  three islands formed by $x$ and two islands formed by $y$.
 We use $x^+$ to denote $x^n$ when $n$ is a positive integer and its exact value is unimportant.
If $\bf u$ is a word over a two-letter alphabet then the {\em height} of $\bf u$ is the number of islands in $\bf u$. For example,
the word $x^+$ has height 1,  $x^+y^+$ has height 2,   $x^+y^+x^+$ has height 3, and so on.
For each $\ell \ge 2$ consider the following property of a semigroup $S$.

$\bullet$ (C$_\ell$) If the height of ${\bf u} \in \{x,y\}^+$ is at most $\ell$, then
$\bf u$ can form an identity of $S$ only with a word of the same type.

The following words were used  by M. Jackson to prove Lemma 5.4 in \cite{MJ}:
\[ {\bf J}_{n} = (x_1 x_{1 + n} \dots x_{1 + n^2 - n})(x_2 x_{2 + n} \dots x_{2+n^2 -n}) \dots  (x_n x_{2n} \dots x_{n^2}), \quad  n>3. \]
For example,
\[ {\bf J}_{4} = (x_1 x_{5} x_9 x_{13})(x_2 x_{6} x_{10} x_{14})(x_3 x_{7} x_{11} x_{15})(x_4 x_{8} x_{12} x_{16}).\]
We generalize Jackson words slightly as follows:
\[ {\bf J}_{n, k} = (x^k_1 x^k_{1 + n} \dots x^k_{1 + n^2 - n})(x^k_2 x^k_{2 + n} \dots x^k_{2+n^2 -n}) \dots  (x^k_n x^k_{2n} \dots x^k_{n^2}),  n>3, k>0.\]

  Notice that the words ${\bf J}_n$ and ${\bf J}_{n,k}$ are of the same type for all $n>3$ and $k>0$.
We use $\overline{{\bf u}}$ to denote the reverse of a word ${\bf u}$.
The following theorem gives us a sufficient condition under which a monoid is NFB  and will be proved in Section \ref{sec:thm}.

\begin{theorem} \label{main}  Let $M$ be a monoid that satisfies Property (C$_5$).
If for each $n>3$, $M$ satisfies the identity

\begin{equation} \label{e1} {\bf U}_n = (x_1 x_2 \dots x_{n^2-1}x_{n^2}) \hskip.04in   {\bf J}_{n, k}  \hskip.04in (x_{n^2} x_{n^2-1} \dots x_2 x_1) \approx\end{equation}
\[(x_1 x_2 \dots x_{n^2-1}x_{n^2}) \hskip.04in   \overline{{\bf J}_{n, k}}    \hskip.04in ( x_{n^2} x_{n^2-1} \dots x_2 x_1)  = {\bf V}_n.\]  for some $k \ge 1$, then  $M$ is NFB.
\end{theorem}

We use $\con({\bf u})$ to denote the set of all variables contained in a word ${\bf u}$.
 An identity ${\bf u} \approx {\bf v}$ is called {\em regular} if  $\con({\bf u}) =  \con({\bf v})$.
The following lemma will be generalized and reversed in Corollary \ref{final}.

\begin{lemma} \label{L51} For each $\ell \ge 2$ the monoid $L_\ell^1$ satisfies Property (C$_\ell$).  In other words,
 if $L_\ell^1 \models {\bf u}\approx {\bf v}$ such that $\con({\bf u}) = \{x,y\}$ and the height of ${\bf u}$ is at most $\ell$, then $\bf v$ is of the same type as $\bf u$.
\end{lemma}

\begin{proof} Since the word $x$ is an isoterm for $L_\ell^1$, the monoid $L_\ell^1$ satisfies only  regular identities. In particular, $\con({\bf v}) =\{x,y\}$. If $\bf v$ is not of the same type as $\bf u$ then consider the substitution  $\Theta: \mathfrak A \rightarrow L_\ell^1$  such that $\Theta(x)=b$ and $\Theta(y)=a$. Then $\Theta({\bf u})$ is a subword of $\underbrace{bababa\cdots}_{\text{length }\ell} \ne 0$ and $\Theta({\bf u}) \ne \Theta({\bf v})$.

Therefore, $\bf v$ must be of the same type as $\bf u$.
\end{proof}

Theorem \ref{main} and Lemma \ref{L51} immediately imply the following.

\begin{cor} \label{L52} Let $M$ be a monoid such that $\var(M)$ contains $L_5^1$.
If for each $n>3$, $M$ satisfies the identity \eqref{e1} for some $k \ge 1$,
then $M$ is NFB.
\end{cor}

The next lemma shows that the identities \eqref{e1}  belong to a wider class of identities satisfied by $L_\ell^1$ for each $\ell \ge 1$.

\begin{lemma} \label{first} Let $k\ge 2$ and  ${\bf X}$ be a word such that $\con({\bf X}) =\{x_1, \dots,x_n\}$ and for each $1\le i \le n$, $occ_{\bf X}(x_i) \ge k-1$. Then for each $n>0$, the monoids
\[ L_{2k}^1 =  \langle a,b,1 \mid aa=a, bb=b, (ab)^{k}=0 \rangle \hskip.3in  \text{and}\]
\[L_{2k+1}^1 =  \langle a,b,1 \mid aa=a, bb=b, (ab)^ka=0 \rangle\]
 satisfy the following identity:

\[ {\bf U}_n = x_1 x_2 \dots x_{n-1}x_{n} \hskip.03in {\bf X} \hskip.03in x_{n} x_{n-1} \dots x_2 x_1 \approx  \]
\[x_1 x_2 \dots x_{n-1}x_{n} \hskip.03in \overline{{\bf X}} \hskip.03in  x_{n} x_{n-1} \dots x_2 x_1 = {\bf V}_n.\]

\end{lemma}

\begin{proof}
 First, notice that each variable appears at least $k+1$ times in ${\bf U}_n$ and ${\bf V}_n$.
Fix some substitution $\Theta: \mathfrak A \rightarrow L^1_{2k}$ ($L^1_{2k+1}$). If for some $1 \le i \le n$, the set  $\con(\Theta(x_i))$ contains both $a$ and $b$ then both  $\Theta({\bf U}_n)$ and $\Theta({\bf V}_n)$ contain $(ab)^{k+1}$ or $(ba)^{k+1}$ as a subword and consequently, both are equal to zero. Therefore, we may assume that for each $1 \le i \le n$ we have $\Theta(x_i) \in \{a,b,1\}$.
To avoid some trivial cases we may also assume that $\Theta (x_1 x_2 \dots x_{n-1}x_{n})$ contains both letters $a$ and $b$. Consider a few cases.

{\bf Case 1}: $\Theta({\bf X})$ starts and ends with the same letter: $a$ or $b$.

In this case,  $\Theta( \overline{{\bf X}}) =  \Theta( {\bf X})$ and consequently,  $\Theta({\bf U}_n) = \Theta({\bf V}_n)$.

{\bf Case 2}:  $\Theta( {\bf X}) = (ab)^m$ for some $m >0$.

In this case,  $\Theta( \overline{{\bf X}}) = (ba)^m$ and consequently,
\[\Theta({\bf U}_n) = (a b)(ab)^{m} (b a) = (a b)( ba)^{m} (b a ) = \Theta({\bf V}_n).\]

{\bf Case 3}:  $\Theta( {\bf X}) = (ba)^m$ for some $m >0$.

This case is dual to Case 2. \end{proof}

Corollary \ref{L52} and Lemma \ref{first} immediately imply the following.

\begin{cor} \label{L5}  Let $M$ be a monoid such that $M$  is contained in  $\var(L_\ell^1)$ for some $\ell \ge 5$
and $\var(M)$ contains $L_5^1 =  \langle a,b,1 \mid aa=a, bb=b, ababa=0 \rangle$. Then $M$  is NFB.
\end{cor}

\section{Identities of monoids that satisfy Property (C$_\ell$)}

If a semigroup $S$ satisfies an identity  ${\bf u} \approx {\bf v}$ we write  $S \models {\bf u} \approx {\bf v}$.

\begin{fact} \label{xx}  If for some $k \ge 1$, a monoid $M$ satisfies Property (C$_{2k}$)  then
the word $x^k$ is an isoterm for $M$.
\end{fact}

\begin{proof}  If $M \models x^k \approx x^r$ for some $r \ne k$ then  $M \models (xy)^k \approx (xy)^r$.
To avoid a contradiction to  Property (C$_{2k}$) we conclude that $x^k$ is an isoterm for $M$.
\end{proof}

\begin{fact} \label{xy}  Let $M$ be a monoid that satisfies Property (C$_2$), then

(i) $xy$ is an isoterm for $M$;

(ii)  if $M \models x^+t \approx {\bf v}$ then ${\bf v} =  x^+t$;

(iii)  if $M \models tx^+ \approx {\bf v}$ then ${\bf v} =  tx^+$;

(iv)  if $M \models x^+tx^+ \approx {\bf v}$ then ${\bf v} =  x^+tx^+$.

\end{fact}

\begin{proof} (i) Since the word $x$ is an  isoterm for $M$ by Fact \ref{xx} and $xy$ can form an identity of $M$ only with a word of the same type, the word $xy$ is also an  isoterm for $M$.

Parts (ii) and (iii) follow immediately from the fact that $t$ is an isoterm for $M$ and Property (C$_2$).
Part (iv) follows immediately from  the fact that $t$ is an isoterm for $M$ and Parts (ii)--(iii).\end{proof}

\begin{fact} \label{xtx}  Let $M$ be a monoid that satisfies Property (C$_3$). If $M \models x^+t_1x^+t_2x^+ \approx {\bf v}$ then ${\bf v} =  x^+t_1x^+t_2x^+$.
\end{fact}

\begin{proof}  If  ${\bf v} \ne  x^+t_1x^+t_2x^+$, then view of Fact \ref{xy}(i),(iv), ${\bf v} = x^+t_1 t_2 x^+$.
If we substitute $y$ for $t_1$ and $t_2$ then we obtain  $M \models x^+yx^+yx^+ \approx x^+yyx^+$.
This contradicts the fact that $M$  satisfies Property (C$_3$).
\end{proof}

 If a variable $t$ occurs exactly once in a word ${\bf u}$ then we say that $t$ is {\em linear} in ${\bf u}$. If a variable $x$ occurs more than once in  ${\bf u}$ then we say that $x$ is {\em non-linear} in ${\bf u}$.  Evidently, $\con({\bf u}) = \lin({\bf u}) \cup \non({\bf u})$ where $\lin({\bf u})$ is the set of all linear variables in $\bf u$ and $\non({\bf u})$ is the set of all non-linear variables in $\bf u$.
A {\em block} of $\bf u$ is a maximal subword of $\bf u$ that does not contain any linear variables of $\bf u$.

\begin{lemma} \label{basic} Let $M$ be a monoid that satisfies Property (C$_3$).
If $M \models {\bf u} \approx {\bf v}$ then

(i) $\lin({\bf u}) = \lin({\bf v})$,  $\non({\bf u}) = \non({\bf v})$  and the order of occurrences of linear variables in $\bf v$  is the same as in $\bf u$;

(ii) the corresponding blocks of $\bf u$ and $\bf v$ have the same content. In other words, if
\[{\bf u} = {\bf a}_0t_1 {\bf a}_1t_2 \dots t_{m-1}{\bf a}_{m-1} t_m {\bf a}_m,\] where $\non({\bf u}) = \con({\bf a}_0 {\bf a}_1 \dots {\bf a}_{m-1}{\bf a}_m)$ and  $\lin({\bf u}) = \{ t_1, \dots, t_m\}$, then
\[{\bf v} = {\bf b}_0t_1 {\bf b}_1t_2 \dots t_{m-1}{\bf  b}_{m-1} t_m {\bf b}_m\]
such that $\con({\bf a}_q) = \con({\bf b}_q)$ for each $0 \le q \le m$.

\end{lemma}

\begin{proof} Part (i) is an immediate consequence from Fact \ref{xy} (i).

 In order  to verify  Part (ii), it is enough to assume that ${\bf u}$ contains exactly one non-linear variable $x$.
In this case, in view of Fact \ref{xy}(ii)--(iv), $\bf u$ and $\bf v$ begin
and end with the same variables. If some corresponding blocks of $\bf u$ and $\bf v$ did not have the same content then for some $1 \le q \le m$, $\bf u$ would contained  $t_qt_{q+1}$ as a subword but $\bf v$ contained $t_q x^+ t_{q+1}$ as a subword or vice versa. Since this  contradicts Fact \ref{xtx}, the corresponding blocks of $\bf u$ and $\bf v$ have the same content.
\end{proof}

If $\con({\bf u}) \supseteq \{x_1, \dots, x_n\}$ we write ${\bf u}(x_1, \dots, x_n)$ to refer to the word obtained from ${\bf u}$ by deleting all occurrences of all variables that are not in $\{x_1, \dots, x_n\}$.

\begin{lemma} \label{begend}  Let $\ell >2$ and $M$ be a monoid that satisfies Property (C$_\ell$).
Let  $\bf u$ be a word with $\non({\bf u}) = \{x, y \}$  such that the height of ${\bf u}(x, y)$ is at most $\ell$.

If $M \models {\bf u} \approx {\bf v}$ then
the corresponding blocks of $\bf u$ and $\bf v$ begin and end with the same variables.
\end{lemma}

\begin{proof}  By Lemma \ref{basic}, we have \[{\bf u} = {\bf a}_0t_1 {\bf a}_1t_2 \dots t_{m-1}{\bf a}_{m-1} t_m {\bf a}_m\]
and \[{\bf v} = {\bf b}_0t_1 {\bf b}_1t_2 \dots t_{m-1}{\bf  b}_{m-1} t_m {\bf b}_m,\]
such that   for each $0 \le q \le m$, $\con({\bf a}_q) = \con({\bf b}_q) \subseteq \{x,y\}$.

Since the height of ${\bf u}(x, y)$ is at most $\ell$, Property  (C$_\ell$)  implies that   ${\bf u}$ and $\bf v$ begin and end with the same variables.  The rest follows from the following.

\begin{claim}  ${\bf u}$ contains a subword $c_1tc_2$ for some $c_1, c_2 \in \{x,y\}$ and $t \in \{t_1, \dots, t_m\}$ if and only if
${\bf v}$ contains the identical 3-letter subword.
\end{claim}

\begin{proof} To obtain a contradiction, suppose that ${\bf u}(x,y,t)$ and ${\bf v}(x,y,t)$ have different 3-letter subwords with $t$ in the middle. Modulo renaming variables and duality there are three cases.

{\bf Case 1}: ${\bf u}$ contains $ytx$ as a subword but $\bf v$ contains $xtx$ as a subword.

In this case, let  $\Theta: \mathfrak A \rightarrow \mathfrak A^+$  be a substitution such that $\Theta(t) =yx$ and is identical on all other variables.
Then $\Theta({\bf u}(x,y,t))$ has the same type as ${\bf u}(x,y)$
 but $\Theta({\bf v}(x,y,t))$ has bigger height than  ${\bf u}(x,y)$. This contradicts Property (C$_\ell$).

{\bf Case 2}: ${\bf u}$ contains $ytx$ as a subword but $\bf v$ contains $xty$ as a subword.

In this case, let  $\Theta: \mathfrak A \rightarrow \mathfrak A^+$  be a substitution such that $\Theta(t) =yx$ and is identical on all other variables.
Then $\Theta({\bf u}(x,y,t))$ has the same type as ${\bf u}(x,y)$
 but $\Theta({\bf v}(x,y,t))$ has bigger height than  ${\bf u}(x,y)$. This contradicts Property  (C$_\ell$).

{\bf Case 3}: ${\bf u}$ contains $yty$ as a subword but $\bf v$ contains $xtx$ as a subword.

In this case, let  $\Theta: \mathfrak A \rightarrow \mathfrak A^+$  be a substitution such that $\Theta(t) =y$ and is identical on all other variables.
Then $\Theta({\bf u}(x,y,t))$ has the same type as ${\bf u}(x,y)$
 but $\Theta({\bf v}(x,y,t))$ has bigger height than  ${\bf u}(x,y)$. This contradicts Property  (C$_\ell$).
\end{proof}
\end{proof}

\begin{lemma} \label{2let} Let  $\ell >2$ and $M$ be a monoid that satisfies Property (C$_\ell$).
Let ${\bf u}$  be a word with  $\non({\bf u}) =  \{x,y\}$ such that

(i)  the height of ${\bf u}(x,y)$ is at most $\ell$;

(ii) every block of $\bf u$  has  height  at most $3$.

Then $\bf u$  can form an identity of $M$ only with a word of the same type.
\end{lemma}

\begin{proof}  We have \[{\bf u} = {\bf a}_0t_1 {\bf a}_1t_2 \dots t_{m-1}{\bf a}_{m-1} t_m {\bf a}_m,\] where $\{x,y\} = \con({\bf a}_0 {\bf a}_1 \dots {\bf a}_{m-1}{\bf a}_m)$ and  $\lin({\bf u}) = \{ t_1, \dots, t_m\}$.

If $M \models {\bf u} \approx {\bf v}$ then by  Lemmas \ref{basic} and  \ref{begend} we have
\[{\bf v} = {\bf b}_0t_1 {\bf b}_1t_2 \dots t_{m-1}{\bf  b}_{m-1} t_m {\bf b}_m\]
such that  for each $0 \le q \le m$, $\con({\bf a}_q) = \con({\bf b}_q) \subseteq \{x,y\}$ and the corresponding blocks ${\bf a}_q$ and ${\bf b}_q$
begin and end with the same variable.

 Condition (ii) implies that for each $0 \le q \le m$,  the block ${\bf a}_q$ is either empty or ${\bf a}_q \in \{x^+, y^+, x^+y^+, y^+x^+, x^+y^+x^+, y^+x^+y^+\}$.
Thus, if for some  $0 \le q \le m$, the corresponding blocks   ${\bf a}_q$ and   ${\bf b}_q$ are not of the same type, then only the following two cases are possible  modulo renaming variables:

{\bf Case 1:}  ${\bf a}_q = x^+y^+$ but  ${\bf b}_q = (x^+y^+)^r$ for some $r>1$;

{\bf Case 2:}  ${\bf a}_q = x^+y^+x^+$ but  ${\bf b}_q = (x^+y^+x^+)^r$ for some $r>1$.

So, if $\bf u$ and $\bf v$ are not of the same type, then some blocks of $\bf v$ have bigger height than the corresponding blocks in $\bf u$.  Therefore, ${\bf v}(x,y)$ has bigger height than  ${\bf u}(x,y)$. To avoid a contradiction to  Property (C$_\ell$), we conclude that $\bf u$ and $\bf v$ are of the same type.\end{proof}

\begin{lemma} \label{correct}  Let  $\ell >2$ and $M$ be a monoid that satisfies Property (C$_\ell$).
Let  $\bf u$ be a word such that

(i) for each  $\{x,y\}\subseteq  \con({\bf u})$ the height of   ${\bf u}(x,y)$ is at most $\ell$;

(ii) for each $x \in \non({\bf u})$  there is a linear variable $t \in  \lin({\bf u})$  between any two islands formed by $x$.

Then $\bf u$  can form an identity of $M$ only with a word of the same type.
\end{lemma}

\begin{proof} We have \[{\bf u} = {\bf a}_0t_1 {\bf a}_1t_2 \dots t_{m-1}{\bf a}_{m-1} t_m {\bf a}_m\] where  $\lin({\bf u}) = \{ t_1, \dots, t_m\}$ and
$\con({\bf a}_0 {\bf a}_1 \dots {\bf a}_{m-1}{\bf a}_m) =  \{x_1, \dots,x_n\} = \non({\bf u})$.

 Condition (ii) implies that for each  $1 \le i \le n$ and each $0 \le q \le m$,  each variable $x_i$ forms at most one island in ${\bf a}_q$.
Lemma \ref{2let} implies that for each $1 \le i < j \le n$, ${\bf u}(x_i, x_j, t_1, \dots,t_m)$ forms an identity of $M$ only with a word of the same type.  Therefore,  $\bf u$  can also form an identity of $M$ only with a word of the same type.
\end{proof}

\section{Words and substitutions}

\begin{lemma} \label{redef}  Let $\bf u$ and $\bf v$ be two words of the same type such that $\lin({\bf u}) = \lin({\bf v})$ and
$\non({\bf u}) = \non({\bf v})$.

Let  $\Theta: \mathfrak A \rightarrow \mathfrak A^+$
be a substitution that has the following property:

(*) If $\Theta(x)$ contains more than one variable then $x$ is  linear  in $\bf u$.

Then $\Theta({\bf u})$ and  $\Theta({\bf v})$  are also of the same type.

\end{lemma}

\begin{proof}
Since $\bf u$ and $\bf v$ are of the same type, for some $r \ge 1$ and $u_1, \dots,u_r, v_1, \dots, v_r >0$ we have ${\bf u} =  c_1^{u_1} c_2^{u_2}\dots c_r^{u_r}$ and  ${\bf v} =  c_1^{v_1}c_2^{v_2}\dots  c_r^{v_r}$,
where $c_1, \dots , c_r$ are not necessarily distinct variables.

 First, let us prove that for each $1 \le i \le r$ the words $\Theta({c}_i^{u_i})$ and $\Theta({c}_i^{v_i})$ are of the same type.
Indeed, if $c_i$ is linear in $\bf u$ (and in $\bf v$) then $u_i=v_i=1$ and $\Theta({c}_i^{u_i}) =  \Theta({c}_i^{v_i})$.
If $c_i$ is non-linear in $\bf u$  (and in $\bf v$) then $\Theta({c}_i^{u_1}) = x^*$ for some variable $x$ and  $\Theta({c}_i^{v_i})$ is a power of the same variable.

Since \[\Theta({\bf u}) = \Theta({c}_1^{u_1}{c}_2^{u_2}\dots {c}_r^{u_r} ) = \Theta({c}_1^{u_1}) \Theta({c}_2^{u_2})\dots \Theta({c}_r^{u_r})\] and
\[\Theta({\bf v}) = \Theta({c}_1^{v_1}{c}_2^{v_2}\dots {c}_r^{v_r} ) = \Theta({c}_1^{v_1}) \Theta({c}_2^{v_2})\dots \Theta({ c}_r^{v_r}),\]
we conclude that $\Theta({\bf u})$ and $\Theta({\bf v})$ are of the same type.
\end{proof}

If $x$ and $y$ are two distinct variables then $E_{x=y}$ denotes a substitution that renames $y$ by $x$ and is identical on all other variables.

\begin{fact} \label{Eu}  Given a word $\bf u$ and a substitution $\Theta: \mathfrak A \rightarrow \mathfrak A^+$, one can equalize some variables in $\bf u$ so that the resulting word $E({\bf u})$ has the following properties:

(i) $\Theta(E({\bf u}))$ is of the same type as  $\Theta ({\bf u})$;

(ii) for every  $x, y \in \con (E({\bf u}))$,  if the words $\Theta(x)$  and  $\Theta(y)$ are powers of the same variable then $x=y$.

\end{fact}

\begin{proof} If $\bf u$ satisfies Property (ii) then take $E$ to be the identity substitution $E_{x=x}$ and we are done.
If  $\bf u$ does not satisfy Property (ii) then for some  $x \ne y \in \con({\bf u})$ the words  $\Theta(x)$  and  $\Theta(y)$ are powers of the same variable. If  $E_{x=y}(\bf u)$ satisfies Property (ii) then take $E=E_{x=y}$. Notice that  $\Theta(E_{x=y}({\bf u}))$ is of the same type as  ${\bf U}$.  If not,  then for some  $p \ne z \in \con(E_{x=y}({\bf u}))$ the words  $\Theta(p)$  and  $\Theta(z)$ are powers of the same variable.  If  $E_{p=z} E_{x=y}(\bf u)$ satisfies Property (ii) then take $E=E_{p=z} E_{x=y}$ and we are done. And so on. Since the number of variables in $E({\bf u})$ decreases, eventually the
word $E({\bf u})$ will satisfy Property (ii).
\end{proof}

\begin{lemma} \label{prop1} Let $\ell >1$ and $\bf U$ be a word such that
 for each  $\{x,y\}\subseteq  \con({\bf U})$ the height of  ${\bf U}(x,y)$ is at most $\ell$.
 Let $\Theta: \mathfrak A \rightarrow \mathfrak A^+$  be a substitution which satisfies Property (ii) in Fact \ref{Eu}.
If $\Theta({\bf u}) = {\bf U}$ then $\bf u$  satisfies Condition (i)  in Lemma \ref{correct},  that is, for each  $\{x,y\}\subseteq  \con({\bf u})$ the height of  ${\bf u}(x,y)$ is at most $\ell$.

\end{lemma}

\begin{proof}  Suppose that for some  $\{x,y\}\subseteq  \con({\bf u})$ the word  ${\bf u}(x,y)$ has height bigger than $\ell$.
Since  $\Theta$ satisfies Property (ii) in Fact \ref{Eu}, $\Theta(x)$ contains $x'$ and  $\Theta(y)$ contains $y'$ for some $x' \ne y'$.
Therefore,   ${\bf U}(x', y')$ also has height  bigger than $\ell$. A contradiction.
\end{proof}

\section{Proof of Theorem \ref{main}}\label{sec:thm}

The following lemma implies \cite[Corollary 2.2]{OS} and is a special case of Fact 2.1 in \cite{OS}.

\begin{lemma} \label{nfblemma} Let $\tau$ be an equivalence relation on the free semigroup $\mathfrak A^+$ and $S$ be a semigroup.
Suppose that for infinitely many $n$, $S$ satisfies an identity ${\bf U}_n \approx {\bf V}_n$ in at least $n$ variables
such that ${\bf U}_n$ and ${\bf V}_n$ are not $\tau$-related.

Suppose also that for every identity ${\bf u} \approx {\bf v}$ of $S$ in less than $n$ variables, every
 word  $\bf U$ such that ${\bf U} \tau {\bf U}_n$ and every substitution
 $\Theta: \mathfrak A \rightarrow \mathfrak A^+$ such that $\Theta({\bf u}) = {\bf U}$ we have
 ${\bf U} \tau \Theta({\bf v})$.  Then $S$ is NFB.
\end{lemma}

\begin{proof} Take an arbitrary $m>0$ and let $\Sigma$ be a set of identities of $S$ in at most $m$ variables.
By our assumption, $S$ satisfies an identity ${\bf U}_n \approx {\bf V}_n$ in at least $n$ variables such that $n>m$ and the words ${\bf U}_n$ and ${\bf V}_n$ are not $\tau$-related.

If  ${\bf U}_{n} \approx {\bf V}_{n}$ was a consequence from $\Sigma$ then one could find a sequence of words ${\bf U}_{n}={\bf W}_1 \approx {\bf W}_2 \approx \dots \approx {\bf W}_l={\bf V}_{n}$ and substitutions $\Theta_1, \dots, \Theta_{l-1} (\mathfrak A \rightarrow \mathfrak A ^+$) such that for each  $i=1, \dots, l-1$ we have ${\bf W}_i=\Theta_i({\bf u}_i)$ and  ${\bf W}_{i+1}=\Theta_i({\bf v}_i)$ for some identity ${\bf u}_i \approx {\bf v}_i \in \Sigma$.
Since every identity in $\Sigma$ involves less than $n$ variables, we have  ${\bf U}_{n}={\bf W}_1 \tau {\bf W}_2 \tau \dots \tau {\bf W}_{l-1} \tau {\bf W}_l={\bf V}_{n}$.
Thus  ${\bf U}_{n} \tau {\bf V}_n$.

Since ${\bf U}_n$ and ${\bf V}_n$ are not $\tau$-related,  ${\bf U}_{n} \approx {\bf V}_{n}$ is not a consequence from $\Sigma$. Since $m$ and $\Sigma$ were arbitrary, $S$ is NFB.
\end{proof}

 Let $\bf U$ be a word of the same type as
${\bf U}_n = x_1x_2 \dots x_{n^2}\hskip .03in  {\bf J}_{n} \hskip.03in x_{n^2} \dots x_2x_1$.
Then the occurrences of $x_{n^2}$ form two islands in $\bf U$. We refer to these two islands as  ${_1x_{n^2}^+}$ and  ${_2x_{n^2}^+}$
counting rightwards from the left. For each $1 \le i < n^2$, the occurrences of $x_{i}$ form three islands in $\bf U$. We refer to these three islands as  ${_1x_{i}^+}$,   ${_2x_{i}^+}$ and  ${_3x_{i}^+}$ counting rightwards from the left.

\begin{lemma} \label{dis} Let $\bf U$ be a word of the same type as
\[{\bf U}_n = x_1x_2 \dots x_{n^2}\hskip .03in {\bf J}_n \hskip .03in x_{n^2} \dots x_2x_1.\]
 Then $\bf U$ has the following properties:

(P1)  for each $1 \le i \ne j \le n^2$ the word $x_i x_j$ appears at most once in $\bf U$ as a subword;

(P2)  for each $1\le i \le n^2$ there are occurrences of at least  $n$  pairwise distinct variables between any two islands formed by $x_i$ in $\bf U$.
\end{lemma}

\begin{proof} Property (P1) is evident.
To verify Property (P2) notice that there are occurrences of  $n^2-1$  pairwise distinct variables between  ${_1x_{n^2}^+}$ and  ${_2x_{n^2}^+}$. If $1 \le i < n^2$ consider two cases.

{\bf Case 1}: $n^2 - i < n$.

In this case, the following $(n-1)$ islands are located  between  ${_1x_{n^2}^+}$ and  ${_2x_i^+}$:
\[\{{_2x^+_1}, {_2x^+_{1+n}}, {_2x^+_{1+2n}}, \dots, {_2x^+_{1+(n-2)n}}\}.\]
Therefore, there are at  least  $n$  pairwise distinct variables between  ${_1x_i^+}$ and ${_2x_i^+}$.

The following $(n-1)$ islands are located between  ${_2x_i^+}$ and  ${_2x_{n^2}^+}$:
\[\{{_2x^+_n}, {_2x^+_{2n}}, {_2x^+_{3n}}, \dots, {_2x^+_{n^2-n}}\}.\]
Therefore, there are at  least  $n$  pairwise distinct variables between  ${_2x_i^+}$ and ${_3x_i^+}$.

{\bf Case 2}: $n^2 - i \ge n$.

In this case, the following $(n-1)$ islands are located  between  ${_1x_{i}^+}$ and  ${_1x_{n^2}^+}$:
\[\{{_1x^+_{i+1}}, {_1x^+_{i+2}}, {_1x^+_{i+3}}, \dots, {_1x^+_{n^2-1}}\}.\]
Therefore, there are at  least  $n$  pairwise distinct variables between  ${_1x_i^+}$ and ${_2x_i^+}$.

The following (n-1) islands are located between  ${_2x_{n^2}^+}$ and  ${_3x_{i}^+}$:
\[\{{_2x^+_{n^2-1}}, {_2x^+_{n^2-2}}, {_2x^+_{n^2-3}}, \dots, {_2x^+_{i+1}}\}.\]
Therefore, there are at  least  $n$  pairwise distinct variables between  ${_2x_i^+}$ and ${_3x_i^+}$.
\end{proof}

\begin{proof}[Proof of Theorem \ref{main}]  Let $\tau$ be the equivalence relation on $\mathfrak A^+$ defined by ${\bf u} \tau {\bf v}$ if $\bf u$ and $\bf v$ are of the same type.
First, notice that the words  ${\bf U}_n$ and  ${\bf V}_n$ are not of the same type. Indeed,  ${\bf U}_n$
contains $x_{n^2} x_1$ as a subword but  ${\bf V}_n$ does not have this subword.

Now let  ${\bf U}$ be of the same type as ${\bf U}_n$.
Let ${\bf u} \approx {\bf v}$ be an identity of $M$ in less than $n$ variables  and let
 $\Theta: \mathfrak A \rightarrow \mathfrak A^+$  be a substitution such that $\Theta({\bf u}) = {\bf U}$.
The word $E({\bf u})$ also involves less than $n$ variables and $E({\bf u}) \approx E({\bf v})$ is also an identity of $M$.

 Since ${\bf U}(x_i, x_j) = x_i^+ x_j^+ x_i^+ x_j^+ x_i^+$  for each  $1\le i < j \le n^2$,
the height of ${\bf U}(x_i, x_j)$ is $5$. So, by Lemma \ref{prop1}, $E({\bf u})$ satisfies Condition (i)  in Lemma  \ref{correct}  for $\ell=5$, that is,  for each  $\{x,y\}\subseteq  \con(E({\bf u}))$, the height of $E({\bf u})(x,y)$ is at most $5$.

If $y \in \non(E({\bf u}))$ then in view of Property (P1) in Lemma  \ref{dis}, $\Theta(y)=x_i^*$ for some $1\le i \le n^2$.
Since the occurrences of $x_i$ form at most three islands in $\bf U$ and $\Theta$ satisfies  Property (ii) in Fact \ref{Eu},
the occurrences of $y$ also form at most three islands in $\bf u$.

Due to Property (P2) in Lemma  \ref{dis},  there are occurrences of at least  $n$  pairwise distinct variables between any two islands formed by $x_i$ in $\bf U$. Since $E({\bf u})$ involves less than $n$ variables, there is a variables $t\in \con(E({\bf u}))$  between
any two islands formed by $y$ in $E({\bf u})$ such that $\Theta(t)$ contains $x_ix_j$ as a subword  for some $1\le i \ne j \le n^2$. Due to Property (P1) in Lemma  \ref{dis}, $t$ is linear in $E({\bf u})$.  Thus $E({\bf u})$ satisfies Condition (ii)  in Lemma  \ref{correct}.
 Therefore, $E({\bf v})$ is  of the same type as $E({\bf u})$ by Lemma \ref{correct}.

Due  to Property (P1) in Lemma  \ref{dis}, $\Theta$ satisfies Condition (*) in  Lemma \ref{redef}.
Consequently, the word $\Theta(E({\bf v}))$ has the same type as  $\Theta(E({\bf u}))$  by Lemma \ref{redef}.
Thus we have
\[ {\bf U} = \Theta({\bf u}) \stackrel{Fact \ref{Eu}}{\tau} \Theta (E ({\bf u})) \stackrel{Lemma \ref{redef}} {\tau} \Theta (E ({\bf v})) \stackrel{Fact \ref{Eu}}{\tau} \Theta({\bf v}).\]
Since $\Theta({\bf v})$ is of the same type as $\bf U$, $M$ is NFB by Lemma \ref{nfblemma}.
\end{proof}

\section{ Sets of isoterms for $L_{\ell}^1$ are FB when $\ell \le 5$}

If $\var S(W) = \var S(W')$ we say that sets of words $W$ and $W'$ are equationally equivalent and write $W \sim W'$.
A word $\bf u$ is called {\em $k$-limited} if each variable occurs in $\bf u$ at most $k$ times.

\begin{fact}  \label{isot}

 (i)$\iso(L_2^1) = \iso(L_3^1) \sim\{ab\}$.

(ii) $\iso(L_4^1) = \iso(L_5^1) \sim\{abab, a^2b^2, ab^2a\}$.

\end{fact}

\begin{proof} First, notice that for each $k \ge 1$, $L_{2k}^1 =  \langle a,b,1 \mid aa=a, bb=b, (ab)^{k}=0 \rangle$ and
$L_{2k+1}^1 =  \langle a,b,1 \mid aa=a, bb=b, (ab)^ka=0 \rangle$
satisfy  $x t_1 x t_2  x \dots x t_k x  \approx  x^2  t_1 x t_2  x \dots x t_k x$.
Therefore, every isoterm for $L_{2k}^1$ and for $L_{2k+1}^1$ is $k$-limited.

(i) Since $L_2^1$  satisfies Property (C$_2$) by Lemma \ref{L51}, the word $xy$ is an isoterm for
 $L_2^1$  by Fact \ref{xy}. Since $\{ab\}$ is equationally equivalent to the set of all $1$-limited words,
we have  $\iso(L_2^1) = \iso(L_3^1) \sim\{ab\}$.

(ii)  Since $L_4^1$  satisfies Property (C$_4$) by Lemma \ref{L51}, the word $x^2$ is an isoterm for
 $L_4^1$  by Fact \ref{xx}.

Let us show that if ${\bf u} \in \{abab, a^2b^2, ab^2a\}$ then $\bf u$ is an isoterm for  $L_4^1$.
Indeed, assume  $L_4^1 \models {\bf u} \approx {\bf v}$.
Since $x^2$ is an isoterm for  $L_4^1$,  the identity ${\bf u} \approx {\bf v}$ is {\em balanced}, that is, every variable occurs the same number of times in $\bf u$ and $\bf v$. Then $\bf v$ can only be one of the words $\{abab, a^2b^2, ab^2a\}$ modulo renaming $a$ and $b$.
To avoid a contradiction to  Property (C$_4$) we conclude that ${\bf v} = {\bf u}$.

 Since $\{abab, a^2b^2, ab^2a\}$ is  equationally equivalent to the set of all $2$-limited words, we have  $\iso(L_4^1) = \iso(L_5^1) \sim\{abab, a^2b^2, ab^2a\}$.
\end{proof}

Notice that the word $xyyxyx$ is 3-limited but is not an isoterm for $L_6^1$ because $L_6^1 \models xyyxyx \approx xyxyyx$.

Since for each $k>0$ the set of all $k$-limited words is FB \cite{JS}, the result of W. Zhang \cite{WTZ} that $L_3^1$ is NFB, Corollary \ref{L5} and Fact \ref{isot} immediately imply the following.

\begin{cor} The monoids $L_{3}^1$ and $L_5^1$ are  NFB while the sets  of their  isoterms  are FB.
\end{cor}

 Presently,  $L_{3}^1$, $L_4^1$\cite{IO} and $L_5^1$ are the only existing examples of NFB finite aperiodic monoids whose sets of isoterms are FB.

\begin{question} Is there a finite aperiodic  FB monoid whose set of isoterms is NFB? Is there a finite aperiodic NFB monoid with
central idempotents whose set of isoterms is FB?
\end{question}

\section{Monoids of the form $S^1_\tau(W)$}

Let $\tau$ be a congruence on the free semigroup $\mathfrak A^+$ and $W$ be a non-empty set of words in $\mathfrak A^+$ such that

$\bullet$ $W$ is a union of $\tau$-classes, that is, ${\bf v} \in W$ whenever  ${\bf u} \in W$ and  ${\bf u} \tau {\bf v}$;

$\bullet$ $W$ is closed under taking subwords, that is, ${\bf v} \in W$ whenever  ${\bf u} \in W$ and  ${\bf v}$ is a subword of ${\bf u}$.

Since $W$ is a union of $\tau$-classes the set $I(W) = \mathfrak A^+ \setminus W$ is also a union of $\tau$-classes if it is not empty.
Let $T$ denote the factor-semigroup of $\mathfrak A^+$ over $\tau$ and $T^1$ denote the monoid obtained by adjoining an identity element to $T$.
Let $H_\tau$ denote the homomorphism corresponding to $\tau$ extended to $\mathfrak A^*$ by $H_\tau(\epsilon) =1$ where $\epsilon$ denotes the empty word and $1$ denotes the identity element of $T^1$.
Since $W$ is closed under taking subwords, $H_\tau(I(W))$ is an ideal of $T = H_\tau(\mathfrak A^+)$ and of $T^1 = H_\tau(\mathfrak A^*)$.
We define $S_\tau(W)$ as the Rees quotient of $T$ over $H_\tau (I(W))$ and $S^1_\tau(W)$ as the Rees quotient of $T^1$ over $H_\tau (I(W))$. Notice that $S_\tau(W) = \phi H_\tau (\mathfrak A^+)$ and $S^1_\tau(W) = \phi H_\tau (\mathfrak A^*)$
where $\phi H_\tau({\bf u}) = H_\tau({\bf u})$ if ${\bf u} \in W \cup \{\epsilon\}$ and  $\phi H_\tau({\bf u}) = 0$  if ${\bf u} \in I(W)$.

If $\tau$ is the trivial congruence on $\mathfrak A^+$ then  $S^1_\tau(W)$ coincides with the widely studied monoid $S^1(W)$ defined in the introduction. Also, recall from the introduction that a word ${\bf u}$ is called a {\em $\tau$-term} for a semigroup $S$ if ${\bf u} \tau {\bf v}$ whenever $S \models {\bf u} \approx {\bf v}$.
The following lemma generalizes Fact \ref{prec}.

\begin{lemma} \label{prec1}
Let $\tau$ be a congruence on the free semigroup $\mathfrak A^+$ such that for each $x \in \mathfrak A$ if $x \tau {\bf u}$ then ${\bf u} = x^m$ for some $m>0$.
Let $W$ be a non-empty set of words in $\mathfrak A^+$ which is a union of $\tau$-classes and is closed under taking subwords.
Let $S$  be a semigroup (resp. monoid).
Then  $\var S$ contains $S_\tau(W)$ (resp. $S^1_\tau(W)$) if and only if every word in $W$ is a $\tau$-term for $S$.
\end{lemma}

\begin{proof}
$\Rightarrow$ Let $S$ be a semigroup such that $\var S$ contains $S_\tau(W)$. Take ${\bf u} \in W$.
Let us show that $\bf u$ is a $\tau$-term for $S$.

Indeed, suppose that $S \models {\bf u} \approx {\bf v}$. Since $\var S$ contains $S_\tau (W)$ we have $\phi H_\tau({\bf u}) =  \phi H_\tau({\bf v})$. Since ${\bf u} \in W$ we have $\phi H_\tau({\bf u}) =  H_\tau({\bf u}) \ne 0$.
If ${\bf v} \not \in W$ then $\phi H_\tau({\bf v}) =0$.
Thus ${\bf v} \in W$ and consequently, we have $H_\tau({\bf u}) = \phi H_\tau({\bf u}) = \phi H_\tau({\bf v}) = H_\tau({\bf v})$.
Thus ${\bf u} \tau {\bf v}$. Therefore $\bf u$ is a $\tau$-term for $S$.

$\Leftarrow$ Let $S$ be a semigroup (resp. monoid) such that  every word in $W$ is a $\tau$-term for $S$.
Let  ${\bf u} \approx {\bf v}$ be an identity of $S$ and
$\Theta$ be a substitution $\mathfrak A \rightarrow S_\tau (W)$ (resp. $\mathfrak A \rightarrow S^1_\tau (W)$). If ${\bf u} = c_1 \dots c_r$ and ${\bf v}=d_1 \dots d_l$ for some not necessarily distinct letters $c_1, \dots, c_r$ and $d_1, \dots, d_l$, then
$\Theta({c_1}) = \phi H_\tau({\bf u}_1), \dots, \Theta(c_r) =  \phi H_\tau({\bf u}_r), \Theta({d_1}) = \phi H_\tau({\bf v}_1), \dots, \Theta(d_l) = \phi H_\tau({\bf v}_l)$ for some not necessarily distinct words
${\bf u}_1, \dots, {\bf u}_r, {\bf v}_1, \dots, {\bf v}_l$ from $\mathfrak A^+$ (resp. from $\mathfrak A^*$). Modulo duality three cases are possible.

{\bf Case 1:} Both ${\bf u}_1 \dots {\bf u}_r$ and ${\bf v}_1 \dots {\bf v}_l$ belong to $I(W)$.

In this case
\[\Theta({\bf u}) = \phi H_\tau ({\bf u}_1 \dots {\bf u}_r) = 0 = \phi H_\tau ({\bf v}_1 \dots {\bf v}_l) = \Theta({\bf v}).\]

{\bf Case 2:} ${\bf u}_1 \dots {\bf u}_r \in W$.

In this case, since $S \models {\bf u} \approx {\bf v}$ we have  $S \models ({\bf u}_1 \dots {\bf u}_r) \approx ({\bf v}_1 \dots {\bf v}_l)$.
Since  ${\bf u}_1 \dots {\bf u}_r$ is a $\tau$-term for $S$ we have $({\bf u}_1 \dots {\bf u}_r) \tau ({\bf v}_1 \dots {\bf v}_l)$.
Since $W$ is a union of $\tau$-classes, ${\bf v}_1 \dots {\bf v}_l \in W$.
Therefore,
\[\Theta({\bf u}) = \phi H_\tau ({\bf u}_1 \dots {\bf u}_r) = H_\tau ({\bf u}_1 \dots {\bf u}_r) =  H_\tau ({\bf v}_1 \dots {\bf v}_l) =
\phi H_\tau ({\bf v}_1 \dots {\bf v}_l) = \Theta({\bf v}).\]

{\bf Case 3:} ${\bf u}_1 = \dots = {\bf u}_r = \epsilon$.

This case is possible only if $S$ is a monoid. In this case, ${\bf u} \approx {\bf v}$ is a regular identity.
(Indeed, if for some $y \in \mathfrak A$ we have $y \in \con({\bf u})$ but $y \not \in \con({\bf v})$,
then $S \models x \approx xy^n$ for some $x \ne y$ and $n>0$. Since $W$ is closed under taking subwords, $x$ is a $\tau$-term for $S$ and consequently, $x \hskip .03in \tau \hskip .03in xy^c$, which is forbidden by our assumption about $\tau$.)
Since  $\con({\bf u}) = \con({\bf v})$ we have ${\bf v}_1 =\dots ={\bf v}_l = \epsilon$. Consequently, $\Theta({\bf u}) = 1 = \Theta ({\bf v})$.

Thus we have proved that every identity of $S$ holds in $S_\tau(W)$ and if $S$ is a monoid then every identity of $S$ holds in $S^1_\tau(W)$.
Therefore, $\var S$ contains $S_\tau(W)$ and if $S$ is a monoid then $\var S$ contains $S^1_\tau(W)$.
\end{proof}

Let $\tau$ be the relation on the free semigroup $\mathfrak A^+$ defined by ${\bf u} \tau {\bf v}$ if and only if $\bf u$ and $\bf v$ are of the same type and let $W_\ell$ be the set of all subwords of $\underbrace{b^+a^+b^+a^+b^+\cdots}_{\text{height } \ell}$.

Then for each $\ell \ge 2$, Lee semigroup $L_\ell$ is isomorphic to $S_\tau(W_\ell)$ and Lee monoid $L_{\ell}^1$ is isomorphic to $S^1_\tau(W_\ell)$. Thus Lemma \ref{prec1} immediately implies the following.

\begin{cor} \label{final} Let $\ell \ge 2$ and $S$ be a semigroup (resp. monoid).
Then $\var S$ contains $L_{\ell}$ (resp. $L_{\ell}^1$) if and only if $S$ satisfies Property (C$_\ell$), that is, every word in $\{x,y\}^+$ of height at most $\ell$ can form an identity of $S$ only with a word of the same type.
\end{cor}

\subsection*{Acknowledgement} The author thanks Edmond Lee for the suggestion to prove that all monoids $L_\ell^1$ are NFB
and for many interesting discussions. The author also thanks an anonymous referee for many helpful comments, Mark Sapir for the encouragement and discussions and Inna Mikhaylova for discussions.

\end{document}